\documentclass[11pt,reqno]{amsart}
\usepackage{amssymb}
\usepackage{amsmath}
\usepackage[active]{srcltx}
\usepackage{t1enc}
\usepackage[latin2]{inputenc}
\usepackage{verbatim}
\usepackage{amsmath,amsfonts,amssymb,amsthm}
\usepackage[mathcal]{eucal}
\usepackage{enumerate}
\usepackage[centertags]{amsmath}
\usepackage{graphics}

\numberwithin{equation}{section}

\newtheorem{theorem}{Theorem}[section]

\newtheorem{lemma}{Lemma}[section]

\newtheorem{definition}{Definition}[section]

\begin{document}
	\author{G. Tutberidze, V. Tsagareishvili, and G. Cagareishvil}
	\title[ Ces\'aro Sums for Lipschitz Derivatives]{The Properties of Ces\'aro General Fourier Sums of Functions with Derivatives of Lipschitz Class Functions }

\address {Vakhtang Tsagareishvili, Ivane Javakhishvili Tbilisi State University, Department of Mathematics, Faculty of Exact and Natural Sciences, Chavchavadze str. 1, Tbilisi 0128, Georgia}
\email{cagare@ymail.com}

\address{Giorgi Tutberidze, The University of Georgia, Viktor Kupradze Institute of Mathematics, 77a Merab Kostava St, Tbilisi 0128, Georgia and Ivane Javakhishvili Tbilisi State University, Faculty of Exact and Natural Sciences, Chavchavadze str. 1, Tbilisi 0128, Georgia}
\email{g.tutberidze@ug.edu.ge, g.tutberidze@tsu.ge}

\address {Giorgi Cagareishvili, Ivane Javakhishvili Tbilisi State University, Department of Mathematics, Faculty of Exact and Natural Sciences, Chavchavadze str. 1, Tbilisi 0128, Georgia}
\email{giorgicagareishvili7@gmail.com}

\thanks{The research of third author is supported by Shota Rustaveli National Science Foundation grant no.  FR-24-698.}

	\date{}
	\maketitle

	\begin{abstract}
	In this paper, we investigate the Ces\'aro means of Fourier series with respect to general orthonormal systems (ONS), when the function \( f \) belongs to a certain differentiable class of functions. 	
    
    It is well known that the membership of a function \( f \not\equiv 0 \) in a differentiable class does not, in general, guarantee the summability of its Fourier series with respect to an arbitrary ONS. 
    
    Therefore, in order for the Fourier series with respect to a given ONS to be summable, one must impose additional conditions on the system functions \( \{\varphi_n\} \). 		
    
    The main objective of this work is to determine such conditions on the functions \( \varphi_n \) of the ONS under which the Ces\'aro means of the Fourier series of any function whose derivative belongs to the Lipschitz class \( \mathrm{Lip}_1 \) are uniformly bounded. 		
    
    The results obtained are sharp in the sense that the conditions cannot be essentially weakened.
	
	\end{abstract}
	
	

\section{Preliminary Definitions and Theorems}
Let \( C_L \) be the class of functions such that \( \frac{df}{dx} \in \mathrm{Lip}_1 \).

The space \( \mathrm{Lip}_1 \) is a Banach space with the norm
\[
\|f\|_{\mathrm{Lip}_1} = \|f\|_{C} + \max_{x \ne y} \frac{|f(x) - f(y)|}{|x - y|},
\]
where  \(C[0,1]\) is a space of continuous functions, with the following norm:
\begin{equation*}
	\left\Vert f\right\Vert _{c}=\max_{x\in [0,1]}\left\vert f\left( x\right)
	\right\vert .
\end{equation*}

Suppose that \( (\varphi_n) \) is an ONS on \([0,1]\), and let \( f \in L_2 \) be an arbitrary function. The numbers
\[
C_n(f) = \int_0^1 f(x) \varphi_n(x)\,dx, \quad n = 1, 2, \dots
\]
are called the Fourier coefficients of the function \( f \) with respect to the ONS \( (\varphi_n) \).

The Fourier series of the function \( f \) is given by the expression:
\begin{equation}
	\sum_{k=1}^{\infty} C_k(f)\, \varphi_k(x).
	\label{2.1}
\end{equation}

The general partial sum of the series \eqref{2.1}, denoted by \( S_n(x, f) \), is given by the expression
\begin{equation}
	S_n(x,f) = \sum_{k=1}^{n} C_k(f)\, \varphi_k(x).
	\label{2.2}
\end{equation}

The Ces\'aro means of order \( \alpha \) of the series \eqref{2.1}, denoted by \( \sigma_n^{\alpha}(x, f) \) (see \cite{alexits1961}, Ch. 2, §1, p. 73; §6, p. 115), are given by the expression
\[
\sigma_n^{\alpha}(x, f) = \frac{1}{A_n^{\alpha}} \sum_{k=0}^{n-1} A_{n-k}^{\alpha}\, a_k\, \varphi_k(x), \]
where  \(\alpha > 0,\) and
\[
A_n^{\alpha} = \frac{(1+\alpha)(2+\alpha)\cdots(n+\alpha)}{n!}.
\]
For \( \alpha > 0 \), the following estimate holds:
\[
A_1 n^{\alpha} < A_n^{\alpha} < A_2 n^{\alpha}.
\]

We define the function \( g_n(x) \) by
\[
g_n(x) = \int_0^x \varphi_n(u)\,du,
\]
and define \( Q_n(u, x) \) as
\[
Q_n(u,x) = \frac{1}{A_n^{\alpha}} \sum_{k=0}^{n-1} A_{n-k}^{\alpha} \, g_k(u) \, \varphi_k(x).
\]

Let us denote
\[
D_{nk}^{\alpha} = \frac{A_n^{\alpha}}{A_{n-k}^{\alpha}}.
\]
Then, the function \( Q_n(u,x) \) can be represented as
\begin{equation}
	Q_n(u,x) = \sum_{k=0}^{n-1} D_{nk}^{\alpha} g_k(u) \varphi_k(x).
	\label{2.3}
\end{equation}

Finally, we define
\begin{equation}
	H_n(x) = \sum_{i=1}^{n-1} \left| \int_0^{\frac{i}{n}} Q_n(u,x)\,du \right|,
	\label{2.4}
\end{equation}

\begin{lemma}\label{lemma1}
	\textit{If \( (\varphi_k) \) is an arbitrary ONS, then}
	\begin{equation}
		\frac{1}{n^2} \sum_{k=1}^{n} \varphi_k^2(x) = O(n^{-\frac{1}{2}}), \quad \text{a.e. on } [0,1].
		\label{2.5}
	\end{equation}
	
	\textit{Indeed, since}
	\[
	\frac{1}{n^2} \sum_{k=1}^{n} \varphi_k^2(x) \le \frac{1}{\sqrt{n}} \sum_{k=1}^{n} k^{-\frac{3}{2}} \varphi_k^2(x),
	\]
	\textit{and}
	\[
	\frac{1}{\sqrt{n}} \sum_{k=1}^{\infty} k^{-\frac{3}{2}} \int_0^1 \varphi_k^2(x)\,dx = \frac{1}{\sqrt{n}} \sum_{k=1}^{\infty} k^{-\frac{3}{2}} < +\infty,
	\]
	\textit{it follows from Levi's theorem that the convergence holds almost everywhere, which proves \eqref{2.5}.}
	
\end{lemma}

\textbf{Notation:} Suppose \( G = [0,1] \setminus F \), where
\[
F = \left\{ x \in [0,1] \,\middle|\, \lim_{n \to \infty} \frac{1}{n^2} \sum_{k=1}^{n} \varphi_k^2(x) = \infty \right\}.
\]

It is easy to show that
\[
|F| = 0 \quad \text{and} \quad |G| = 1.
\]

\medskip

\begin{definition}
	We say that \( f \in E(x) \) if, at the point \( x \in [0,1] \),
	\[
	\limsup_{n \to \infty} \left| \sigma_n^{\alpha}(x, f) \right| < +\infty.
	\]
\end{definition}

\begin{lemma} 
	\label{Lemma2} 
	For every integer \( i \) with \( 1 \leq i \leq n \) and for all \( x \in [0,1] \), the following inequality (see \eqref{2.3}):
	\begin{equation}
		\int_{\frac{i-1}{n}}^{\frac{i}{n}} |Q_n(u, x)|\,du \leq \left( \frac{1}{n^2} \sum_{k=1}^n \varphi_k^2(x) \right)^{\frac{1}{2}}.
		\label{2.6}
	\end{equation}
\end{lemma}

\begin{proof}
By the Bessel inequality, it follows that

	\begin{equation}
		\sum_{k=1}^\infty g_k^2(u) = \sum_{k=1}^\infty \left( \int_0^u \varphi_k(t)\,dt \right)^2 \le \int_0^u dt \le 1.
		\label{2.7}
	\end{equation}
Here, the function \( g_k \) is defined by
\[
g_k(u) = \int_0^u \varphi_k(t)\, dt, \quad u \in [0,1].
\]
	
	By applying the Cauchy and H\"older inequalities, as well as inequality \eqref{2.7}, and taking into account that \( D_{nk}^{\alpha} < 1 \), we conclude, for each \( i = 1, \dots, n \),
    
    \begin{align*}
        &\int_{\frac{i-1}{n}}^{\frac{i}{n}} |Q_n(u,x)|\,du
	\le \frac{1}{\sqrt{n}} \left( \int_{\frac{i-1}{n}}^{\frac{i}{n}} Q_n^2(u,x)\,du \right)^{\frac{1}{2}} \\
    &= \frac{1}{\sqrt{n}} \left( \int_{\frac{i-1}{n}}^{\frac{i}{n}} \left( \sum_{k=1}^n D_{nk}^{\alpha} g_k(u) \varphi_k(x) \right)^2 du \right)^{\frac{1}{2}}\\
    &\le \frac{1}{\sqrt{n}} \left( \int_{\frac{i-1}{n}}^{\frac{i}{n}}  \sum_{k=1}^n (D_{nk}^{\alpha})^2 g_k^2(u)\, du \cdot \sum_{k=1}^n \varphi_k^2(x) \right)^{\frac{1}{2}} \\
    &\le \frac{1}{n} \left( \sum_{k=1}^n \varphi_k^2(x) \right)^{\frac{1}{2}}.
    \end{align*}
Hence, Lemma \ref{Lemma2} is proved.
	
\end{proof}

\begin{lemma}
	\label{lemma3}
Let \( (\varphi_n) \) be an ONS on \([0,1]\), and let \( f \in C_L \). Then
	\begin{equation}
		\sigma_n^{\alpha}(x, f) = f(1)\, \sigma_n^{\alpha}(x, q) - \int_0^1 f'(u)\, Q_n(u,x)\,du,
		\label{2.8}
	\end{equation}
	where \( q(u) = 1 \) for all \( u \in [0,1] \).
	
\end{lemma}

\begin{proof}
	Integrating by parts, we have
	\[
	C_n(f) = \int_0^1 f(u)\, \varphi_n(u)\,du
	= f(1) \int_0^1 \varphi_n(u)\,du - \int_0^1 f'(u)\, g_n(u)\,du,
	\]
	that is,
	\begin{equation}
		C_n(f) = f(1) C_n(q) - \int_0^1 f'(u)\, g_n(u)\,du.
		\label{2.9}
	\end{equation}
	It follows that,
	\[
	\sum_{k=0}^{n-1} D_{nk}^{\alpha} C_k(f)\, \varphi_k(x)
	= f(1) \sum_{k=0}^{n-1} D_{nk}^{\alpha} C_k(q)\, \varphi_k(x)
	- \int_0^1 f'(u) \sum_{k=0}^{n-1} D_{nk}^{\alpha} g_k(u) \varphi_k(x)\, du.
	\]
	By taking equation \eqref{2.3} into consideration, we obtain the desired result \eqref{2.8}. 
	
\end{proof}

\begin{theorem}[S. Banach \cite{banach1940}]  \label{Banach}
\textit{For any \( f \in L_2 \),  \( \left(f \not\sim 0 \right) \), there exists an ONS \( (\varphi_n) \) such that}
\[
\limsup_{n \to \infty} \left| \sigma_n^{\alpha}(x, f) \right| = +\infty, \quad \text{a.e. on } [0,1].
\]
	
\end{theorem}
\medskip

\begin{theorem}[see \cite{tsagareishvili2020}]  
\textit{If \( f, F \in L_2 \), then the following identity holds:}
\begin{align}
	\int_0^1 f'(u)F(u)\,du & = n \sum_{i=1}^{n-1} \left( \int_{\frac{i-1}{n}}^{\frac{i}{n}} \left(f(u) - f\left(u +\frac{1}{n}\right)\right) \, du \int_0^{\frac{i}{n}} F(u) \, du \right) \notag \\
	                       & \quad + n \sum_{i=1}^{n-1} \int_{\frac{i-1}{n}}^{\frac{i}{n}} \int_{\frac{i-1}{n}}^{\frac{i}{n}} (f(u) - f(v)) \, dv \, F(u) \, du \notag                            \\
	                       & \quad + n \int_{1 - \frac{1}{n}}^1 f(u)\,du \int_0^1 F(u)\,du.
	\label{2.10}
\end{align}
\end{theorem}

\section{Statement of the Main Problem}
As is well known, the convergence or summability of Fourier series with respect to classical ONS -- such as the trigonometric system (see \cite[Ch. 4, \#1, p. 150]{kashin1999}), the Haar system (see \cite[Ch. 3, \#1, p. 70]{kashin1999}), and the Walsh system (see \cite[Ch. 4, \#5, p. 150]{kashin1999}) -- for differentiable functions is a relatively straightforward matter. This is primarily due to the well-understood structure and good localization properties of these systems, which align well with the smoothness of the functions under consideration.

However, the situation is markedly different when dealing with general ONS. In such cases, classical techniques may fail to guarantee convergence, and even summability can become a subtle and complex issue. The lack of smoothness, poor localization, or irregularity in the structure of the system often leads to new challenges, requiring refined tools and more delicate analysis to obtain meaningful results.

Indeed, as shown by Banach's classical result \cite{banach1940}, the Fourier series with respect to certain ONS may fail to be summable even for simple functions such as \( f(x) = 1 \) on the interval \([0,1]\). This highlights the fact that, in general, summability cannot be taken for granted, even in the case of very regular or constant functions.

In this paper, we identify those ONS \((\varphi_n)\) for which the condition
\[
\sigma_n^\alpha(x,f) = O(1)
\]
holds for every function \( f \) whose derivative \( f' \) belongs to the Lipschitz class \(\mathrm{Lip}1\), at a given point \( x \in G \).

Similar questions have been investigated in a number of recent works (see \cite{tsagareishvili2022a, tsagareishvili2022b, person2023, tsagareishvili2023b, tsagareishvili2022c}), where related aspects of summability and convergence for various ONS were considered.

\section{Main Results}
\begin{theorem}\label{theorem1}
	Let \( (\varphi_n) \) be an ONS on the interval \([0,1]\), and let \( p, q \in E(x) \), where  
	\( p(u) = u \) and \( q(u) = 1 \) for all \( u \in [0,1] \).
	If, for a point \( x \in G \), the condition  
\begin{equation}
	H_n(x) = O(1)
	\label{eq:3.1}
\end{equation}
  holds,  
  then for any function \( f \in C_L \), the following estimate is valid at the point \( x \in G \):
  	\[
	\sigma_n^{\alpha}(x, f) = O(1).
	\]
	
\end{theorem}

\begin{proof}
	If we substitute \( F(x) = Q_n(u,x) \) and \( f = f' \) into equation \eqref{2.10}, we obtain
		\begin{align}
		\int_0^1 f'(u) Q_n(u, x) \, du &= 
		n \sum_{i=1}^{n-1} \left( \int_{\frac{i-1}{n}}^{\frac{i}{n}} \left(f'(u) - f'\left(u + \frac{1}{n} \right) \right) du 
		\int_0^{\frac{i}{n}} Q_n(u, x) \, du \right) \notag \\
		&\quad + n \sum_{i=1}^{n-1} \int_{\frac{i-1}{n}}^{\frac{i}{n}} \int_{\frac{i-1}{n}}^{\frac{i}{n}} 
		\left(f'(u) - f'(v)\right) dv \, Q_n(u, x) \, du \notag \\
		&\quad + n \int_{1 - \frac{1}{n}}^1 f'(u) \, du \int_0^1 Q_n(u, x) \, du.
		\label{eq:3.2}
	\end{align}
	Due to \eqref{eq:3.1} and the fact that  \(f'\in Lip_1,\) we obtain
	\begin{align*}
		|I_1| &= n \cdot O\left(\tfrac{1}{n}\right) \sum_{i=1}^{n-1} \int_{\frac{i-1}{n}}^{\frac{i}{n}} du \left| \int_0^{\frac{i}{n}} Q_n(u, x) \, du \right|  \\
		&= \frac{O(1)}{n} \sum_{i=1}^{n-1} \left| \int_0^{\frac{i}{n}} Q_n(u, x) \, du \right|  = O(1)	T_n(x) = O(1).
	\end{align*}
	
	Next, since \( f' \in \mathrm{Lip}_1 \) and \( x \in G \), by applying inequality \eqref{2.7} from Lemma \ref{lemma1}, which proves that \( D_{nk}^\alpha < 1 \), we obtain:	
	\begin{align*}
		|I_2| 
		&\leq n \cdot O\left( \tfrac{1}{n^2} \right) \sum_{i=1}^n \int_{\frac{i-1}{n}}^{\frac{i}{n}} |Q_n(u, x)| \, du \notag \\
		&= O(1) \cdot \tfrac{1}{n} \int_0^1 |Q_n(u, x)| \, du = O(1) \cdot \tfrac{1}{n} \left( \int_0^1 Q_n^2(u, x) \, du \right)^{\frac{1}{2}} \notag \\
		&= O(1) \cdot \tfrac{1}{n} \left( \int_0^1 \left( \sum_{k=1}^n D_{nk}^{\alpha} g_k(u) \varphi_k(x) \right)^2 du \right)^{\frac{1}{2}} \notag \\
		&= O(1) \cdot \tfrac{1}{n} \left( \sum_{k=1}^n (D_{nk}^{\alpha})^2 \varphi_k^2(x) \int_0^1 g_k^2(u) \, du \right)^{\frac{1}{2}} \notag \\
		&\leq O(1) \cdot \tfrac{1}{n} \left( \sum_{k=1}^n \varphi_k^2(x) \right)^{\frac{1}{2}} = O(1) \cdot \left( \tfrac{1}{n^2} \sum_{k=1}^n \varphi_k^2(x) \right)^{\frac{1}{2}} = O(1) \notag .
	\end{align*}
	Furthermore, according to equation \eqref{2.8}, if 
	\[
	F(u) = p(u) = u, \quad u \in [0,1],
	\]
	then we get
	\[
	\sigma_n^\alpha(x,p) = f(1) \, \sigma_n^\alpha(x,q) - \int_0^1 Q_n(u,x) \, du
	\]
	Now, according to the conditions of Theorem \ref{Banach}, where \( p, q \in E(x) \) or	\(
	\sigma_n^\alpha(x,p) = O(1)  \text{ and } \sigma_n^\alpha(x,q) = O(1), \) 	we derive
	\[
	\left| \int_0^1 Q_n(u,x) \, du \right| = O(1).
	\]
	
Noting that \( f'(u) \in \mathrm{Lip}_1 \), the term \( I_3 \) can be estimated as
\[
|I_3| \leq n \left| \int_{1 - \frac{1}{n}}^{1} f'(u) \, du \right| \cdot \left| \int_0^1 R_n(u,x) \, du \right| = O(1).
\]

Therefore, given that \(|I_1| = O(1)\), \(|I_2| = O(1)\), and \(|I_3| = O(1)\), it follows from equation \eqref{eq:3.2} that
\[
\left| \int_0^1 f'(u) Q_n(u,x) \, du \right| = O(1)
\]
is fulfilled.

Once again, using equation \eqref{2.8} for \( f \in C_L \), it follows that
\[
\sigma_n^\alpha(x,f) = f(1) \, \sigma_n^\alpha(x,q) - \int_0^1 f'(u) Q_n(u,x) \, du,
\]
where \( p(u) = u \) and \( q(u) = 1 \), for \( u \in [0,1] \).

From the last equality, we have
\[
\sigma_n^\alpha(x,f) = O(1).
\]
Hence, Theorem \ref{theorem1} is proved.
	
\end{proof}

\begin{theorem} \label{theorem2}
	Let \( (\varphi_n) \) be an ONS on the interval \([0,1]\).  
	If for some point \( x_0 \in G \) the following holds:
	\[
	\limsup_{n \to \infty} |H_n(x_0)| = +\infty,
	\]
	then there exists a function \( r \) such that \( r' \in \mathrm{Lip}_1 \) and
	\begin{equation} \label{eq:3.3}
		\limsup_{n \to \infty} \left| \sigma_n^\alpha(x_0, r) \right| = +\infty.
	\end{equation}
\end{theorem}

\begin{proof}
	Firstly, we observe that
	\begin{equation} \label{eq3.4}
		\sigma_n^\alpha(x_0, p) = O(1),
	\end{equation}
	and
	\begin{equation} \label{eq3.5}
		\sigma_n^\alpha(x_0, q) = O(1).
	\end{equation}
	Otherwise, since \( p, q \in C_L \), Theorem \ref{theorem2} would already be proved.
	
	On the space \( \mathrm{Lip}_1 \), we consider the sequence of functions \( (r_n) \), defined as follows:
	
	\begin{equation}\label{eq3.6}
		r_n(u) = \int_0^u \operatorname{sign} \left( \int_0^y Q_n(v, x_0) \, dv \right) dy, \quad n = 1, 2, \ldots
	\end{equation}
	
	The following equality holds (see \cite{gogoladze2012}):
	\begin{align}
		\int_0^1 f(u) G(u) \, du &= \sum_{i=1}^{n-1} \left( f\left( \frac{i}{n} \right) - f\left( \frac{i+1}{n} \right) \right) \int_0^{\frac{i}{n}} G(u) \, du \notag \\
		&\quad + \sum_{i=1}^{n-1} \int_{\frac{i-1}{n}}^{\frac{i}{n}} \left( f(u) - f\left( \frac{i}{n} \right) \right) G(u) \, du \notag \\
		&\quad + f(1) \int_0^1 G(u) \, du \label{**}
	\end{align}
	Replacing \(f(u)\) by \(r_n(u)\) and \(G (u) \) by \(Q_n(u, x_0)\) in the equality \eqref{**}, we obtain
	\begin{align}\label{eq3.7}
		\int_0^1 r_n(u) Q_n(u, x_0) \, du 
		&= \sum_{i=1}^{n-1} \left( r_n\left(\frac{i}{n}\right) - r_n\left(\frac{i+1}{n}\right) \right) \int_0^{\frac{i}{n}} Q_n(u, x_0) \, du \notag \\
		&\quad + \sum_{i=1}^{n-1} \int_{\frac{i-1}{n}}^{\frac{i}{n}} \left( r_n(u) - r_n\left(\frac{i}{n}\right) \right) Q_n(u, x_0) \, du \notag \\
		&\quad + r_n(1) \int_0^1 Q_n(u, x_0) \, du=S_1+S_2+S_3.
	\end{align}

Next, we proceed to estimate the quantities \( |S_1| \), \( |S_2| \), and \( |S_3| \).
	
By equation \eqref{eq3.6} and applying Lemma \ref{lemma3}, and in view of the previously established bound on \( |I_2| \) as well as the condition \( D_{nk}^\alpha < 1 \), we arrive at the following relation involving the functions \( Q_n(u, x_0) \) and \( r_n(u) \):
	\begin{align}	\label{eq:3.8}
		|S_2| &\leq \frac{1}{n} \sum_{i=1}^n \int_{\frac{i-1}{n}}^{\frac{i}{n}} |Q_n(u, x_0)| \, du = \frac{1}{n} \int_0^1 |Q_n(u, x_0)| \, du \notag \\
		&= O(1)\frac{1}{n} \left( \int_0^1 \left( \sum_{k=1}^n D_{nk}^\alpha g_k(u) \varphi_k(x_0) \right)^2 du \right)^{\frac{1}{2}} \notag \\
		&= O(1)\frac{1}{n} \left( \int_0^1 \sum_{k=1}^n g_k^2(u) \, du \cdot \sum_{k=1}^n \varphi_k^2(x_0) \right)^{\frac{1}{2}} = O(1).	
	\end{align}
	
	Afterwards, considering equation \eqref{2.8}, and setting \(f'(u) = r_n(u)\), by using the estimates \eqref{eq3.4} and  \eqref{eq3.5}, we obtain
	\begin{equation}\label{eq:3.9}
		|S_3| = O(1) \cdot \left| \int_0^1 Q_n(u, x_0) \, du \right| = O(1).
	\end{equation}
	
	By \( E_n \) we denote the set of all indices \( i \), \( i \in \{1, 2, \ldots, n-1\} \), for which there exists a point \( t \in \left[\frac{i}{n}, \frac{i+1}{n}\right) \) such that
	\[
	\operatorname{sign} \int_0^{\frac{i}{n}} Q_n(u, x_0) \, du \neq \operatorname{sign} \int_0^t Q_n(u, x_0) \, du.
	\]
	
	If \( i \in E_n \), then due to the continuity of the function	\( \int_0^t Q_n(u, x_0) \, du\) on the interval \([0,1]\), there exists some \( t_{in} \in \left[\frac{i}{n}, \frac{i+1}{n}\right) \) such that
	\[
	\int_0^{t_{in}} Q_n(u, x_0) \, du = 0.
	\]
	
	Consequently, we have
	\[
	\int_0^{\frac{i}{n}} Q_n(u, x_0) \, du = \int_{t_{in}}^{\frac{i}{n}} Q_n(u, x_0) \, du.
	\]
	
For this reason, by \eqref{2.7}, we have
\begin{align*}
	&\sum_{i \in E_n} \left| \int_0^{\frac{i}{n}} Q_n(u, x_0) \, du \right|
	\leq \sum_{i=1}^{n-1} \left| \int_{t_{in}}^{\frac{i}{n}} Q_n(u, x_0) \, du \right| \leq \int_0^1 |Q_n(u, x_0)| \, du \notag \\
	&= \left( \int_0^1 Q_n^2(u, x_0) \, du \right)^{\frac{1}{2}} = \left( \int_0^1 \left( \sum_{k=1}^n D_{nk}^\alpha g_k(u) \varphi_k(x_0) \right)^2 du \right)^{\frac{1}{2}} \notag \\
	&= \left( \int_0^1 \sum_{k=1}^n g_k^2(u) \, du \cdot \sum_{k=1}^n \varphi_k^2(x_0) \right)^{\frac{1}{2}} \leq \left( \sum_{k=1}^n \varphi_k^2(x_0) \right)^{\frac{1}{2}}.
\end{align*}
From Lemma \ref{lemma1}, we conclude that 
	\begin{align}\label{eq:3.10}
		\frac{1}{n} \sum_{i \in E_n} \left| \int_0^{\frac{i}{n}} Q_n(u, x_0) \, du \right| = O(1)\frac{1}{n} \left( \sum_{k=1}^n \varphi_k^2(x_0) \right)^{\frac{1}{2}} = O(1).
	\end{align}
	
Let us define the set \( F_n := \{1, 2, \ldots, n-1\} \setminus E_n \). Then, for any \( i \in F_n \), by the definition of \( E_n \), the following relation holds:
\[
\left( r_n\left(\frac{i}{n}\right) - r_n\left(\frac{i+1}{n}\right) \right) \int_0^{\frac{i}{n}} Q_n(u, x_0) \, du = -\frac{1}{n} \left| \int_0^{\frac{i}{n}} Q_n(u, x_0) \, du \right|.
\]
Therefore, using \eqref{eq3.6}, it immediately follows that
\begin{align}\label{eq:3.11}
	&\left| \sum_{i \in F_n} \left( r_n \left(\frac{i}{n}\right) - r_n \left(\frac{i+1}{n}\right) \right) \int_0^{\frac{i}{n}} Q_n(u, x_0) \, du \right|
	= \frac{1}{n} \sum_{i \in F_n} \left| \int_0^{\frac{i}{n}} Q_n(u, x_0) \, du \right| \notag \\
	&= \frac{1}{n} \sum_{i=1}^{n-1} \left| \int_0^{\frac{i}{n}} Q_n(u, x_0) \, du \right| - \frac{1}{n} \sum_{i \in E_n} \left| \int_0^{\frac{i}{n}} Q_n(u, x_0) \, du \right| \notag \\
	&= H_n(x_0) - \frac{1}{n} \sum_{i \in E_n} \left| \int_0^{\frac{i}{n}} Q_n(u, x_0) \, du \right|.
\end{align}
	On the other hand, if \( i \in E_n \), then (see \eqref{eq:3.10}) we have
	\begin{align}\label{eq:3.12}
		&\left| \sum_{i \in E_n} \left( r_n \left(\frac{i}{n}\right) - r_n \left(\frac{i+1}{n}\right) \right) \int_0^{\frac{i}{n}} Q_n(u, x_0) \, du \right| \\
		&\leq \frac{1}{n} \sum_{i \in E_n} \left| \int_0^{\frac{i}{n}} Q_n(u, x_0) \, du \right| = O(1).\notag
	\end{align}
	Afterwards, by combining \eqref{eq:3.11} and \eqref{eq:3.12} , we obtain
	\begin{equation}\label{eq:3.13}
		|S_1| = \left| \sum_{i=1}^{n-1} \left( r_n\left(\frac{i}{n}\right) - r_n\left(\frac{i+1}{n}\right) \right) \int_0^{\frac{i}{n}} Q_n(u, x_0) \, du \right| \geq H_n(x_0) - O(1).
	\end{equation}
Lastly, combining equation \eqref{eq3.7} with \eqref{eq:3.8}, \eqref{eq:3.9}, and \eqref{eq:3.11}, we obtain
\begin{equation*}
	\left| \int_0^1 r_n(u) Q_n(u, x_0) \, du \right| \geq H_n(x_0) - O(1).
\end{equation*}
Keeping in mind the condition of Theorem \ref{theorem2}, it follows that
\begin{equation}\label{eq:3.14}
	\limsup_{n \to \infty} \left| \int_0^1 r_n(u) Q_n(u, x_0) \, du \right| = +\infty.
\end{equation}

Suppose
\[
U_n(f) = \int_0^1 f(u) Q_n(u, x_0) \, du.
\]
Then \( U_n (f) \) is a sequence of linear and bounded functionals on the Banach space \(\mathrm{Lip}_1\). From \eqref{eq:3.14}, it follows that
\[
\limsup_{n \to \infty} |U_n(f_n)| = +\infty.
\]

It is easy to prove that
\begin{equation}\label{eq:3.15}
	\|f_n\|_{\mathrm{Lip}_1} = \|f_n\|_C + \sup_{x,y \in [0,1]} \frac{|f_n(x) - f_n(y)|}{|x - y|} \leq 2.
\end{equation}
According to the Banach--Steinhaus Theorem (see \eqref{eq:3.14} and \eqref{eq:3.15}), there exists a function \( h \in \mathrm{Lip}_1 \) such that
\begin{equation}\label{eq:3.16}
	\limsup_{n \to \infty} \left| \int_0^1 h(u) Q_n(u, x_0) \, du \right| = +\infty.
\end{equation}

Let us define the function \( r (u) \) as follows:
\[
r(u) = \int_0^u h(v) \, dv.
\]
Using \eqref{2.8} for \( f = r \) and \( x = x_0 \), we obtain
	\[
	\sigma_n^{\alpha}(x_0, r) = f(1) \, \sigma_n^{\alpha}(x_0, q) - \int_0^1 h(u) Q_n(u, x_0) \, du.
	\]
Bearing in mind \eqref{eq3.5} and \eqref{eq:3.16}, we conclude that
\[
\limsup_{n \to \infty} \left| \sigma_n^{\alpha}(x_0, r) \right| = +\infty.
\]
Since \( r' = h \in \mathrm{Lip}_1 \), Theorem \ref{theorem2} is proved.

\end{proof}

Now, we demonstrate that the condition in Theorem \ref{theorem2}, namely \( q, p \in E(x) \) with \( q(u) = 1 \) and \( p(u) = u \) for \( u \in [0,1] \), does not necessarily ensure the boundedness of the sequence \( \{ \sigma_n^{\alpha}(x,f) \} \) for every function \( f \in C_L \). More precisely, it is not guaranteed that
\[
\limsup_{n \to \infty} \left| \sigma_n^{\alpha}(x,f) \right| < +\infty.
\]
Indeed,
\begin{theorem}
	\label{theorem3}
	There exists a function \(\gamma \in C_L\) and an ONS \(\{\vartheta_n\}\) such that
	\[
	\int_0^1 \vartheta_n(u) \, du = 0, \quad \int_0^1 u \, \vartheta_n(u) \, du = 0, \quad n = 1,2,\ldots,
	\]
	and
	\begin{align*}
			\limsup_{n \to \infty} \left| \sigma_n^{\alpha}(x, \gamma, \vartheta) \right| = \limsup_{n \to \infty} \left| \sum_{k=1}^n C_k(\gamma, \vartheta) \vartheta_k(x) \right| = +\infty, \quad \text{a.e. } x \in [0,1], 
	\end{align*}
	where
\begin{align*}
	C_k(v, \vartheta) = \int_0^1 \gamma(u) \vartheta_k(u) \, du, \quad n = 1,2,\ldots,
\end{align*}
	and
\begin{align*}
		\sigma_n^{\alpha}(x, \gamma, \vartheta) = \sum_{k=1}^n D_{nk}^\alpha C_k(\gamma, \vartheta) \vartheta_k(x).  
\end{align*}

\end{theorem}

Before proceeding to the proof of Theorem \ref{theorem3}, we first state and prove the following auxiliary lemma, which plays a crucial role in establishing the main result.

\begin{lemma}\label{lemma3.1}
	Let \( f \) be a function and \( (\varphi_n) \) an ONS on \([0,1]\), such that
	
	\begin{itemize}
		\item[(a)] \(\displaystyle \int_0^1 \varphi_n(u) \, du = 0, \quad n = 1, 2, \ldots\)
		\item[(b)] \(\displaystyle \limsup_{n \to \infty} \left| \sigma_n^\alpha(x, f) \right| = +\infty, \quad \text{a. e. } x \in \left[0, \frac{1}{2}\right)\).
	\end{itemize}
	
Then there exist an ONS \((g_n)\) and a function \(V(u)\), for which the following conditions are fulfilled:
	
	\begin{itemize}
		\item[(c)] \(\displaystyle \int_0^1 u g_n(u) \, du = 0, \quad n = 1, 2, \ldots\)
		\item[(d)] \(\displaystyle \limsup_{n \to \infty} \left| \sigma_n^\alpha(x, V) \right| = +\infty, \quad \text{for some } x \in \left[0, \frac{1}{2}\right).\)
	\end{itemize}
	
\end{lemma}

\begin{proof}
We define the ONS \((g_n)\) as follows:
\[
g_n(u) = 
\begin{cases}
	\varphi_n(2u), & u \in \left[0, \frac{1}{2}\right), \\
	-\varphi_n\bigl(2(u - \frac{1}{2})\bigr), & u \in \left[\frac{1}{2}, 1\right].
\end{cases}
\]
It is straightforward to verify that \((g_n)\) forms an ONS on \([0,1]\).

Now consider the function \( v \)
\[
v(u) = \begin{cases}
	f(2u), & u \in [0, \frac{1}{2}), \\
	0, & u \in [\frac{1}{2}, 1].
\end{cases}
\]
In such case we have
\begin{align}\label{3.17}
	C_n (v,g) &= \int_0^1 v(u) g_n (u) \, du = \int_0^{\frac{1}{2}} f(2u) \, \varphi_n (2u) \, du \\
	&= \frac{1}{2} \int_0^1 f(u) \varphi_n (u) \, du = \frac{1}{2} C_n (f).\notag
\end{align}

Next, we verify that the sequence \((g_n)\) satisfies the required orthonormality and moment conditions. Specifically, we compute the integral

\begin{align*}
	\int_0^1 u g_n(u) \, du &= \int_0^{\frac{1}{2}} u \varphi_n(2u) \, du - \int_0^{\frac{1}{2}} u \varphi_n \big( 2(u - \frac{1}{2}) \big) \, du \\
	&= \frac{1}{4} \int_0^1 u \varphi_n(u) \, du - \frac{1}{4} \int_0^1 u \varphi_n(u) \, du + \frac{1}{4} \int_0^1 \varphi_n(u) \, du = 0, \quad n=1,2, \ldots . 
\end{align*}
Finally, from \eqref{3.17} and condition \textbf{b)}, when \( x \in \left[0, \tfrac{1}{2} \right) \), we obtain
\[
\limsup_{n \to \infty} \left| \sigma_n^{\alpha}(x, v) \right| = \frac{1}{2} \limsup_{n \to \infty} \left| \sigma_n^{\alpha}(x, f) \right| = +\infty.
\]
Thus, the lemma is proved.

\end{proof}

After proving the supporting lemma and verifying the necessary conditions, we can begin the proof of Theorem \ref{theorem3}.

\begin{proof}[ Proof of Theorem \ref{theorem3}]
	We assume that  
	\[
	f(u) = 1 - \cos\big(4\pi(u - \frac{1}{2})\big).
	\]
	
	According to the theorem of Banach (see Theorem \ref{Banach}), there exists an ONS  \((\delta_n)\) such that, a. e. on  \([0,1]\),
	\[
	\limsup_{n \to \infty} \left| \sigma_n^{\alpha}(x, f, \delta) \right| = +\infty.
	\]
Now, let define ONS \((\varphi_n)\)  as follows:
        \[
	\varphi_n(u) = 
	\begin{cases}
		\delta_n(2u), & \text{ when } u \in \left[0, \tfrac{1}{2}\right), \\
		-\delta_n \left(2\left(u-\frac{1}{2} \right) \right), & \text{when } u \in \left[\tfrac{1}{2}, 1\right].
	\end{cases}
	\]              
It is easy to prove that \((\varphi_n)\)  is an ONS on \([0,1]\) and \(\int_0^1\varphi_n (u)du = 0, \ \ n=0, 1, \dots .\)

Thus, the ONS \((\varphi_n)\) satisfies the conditions a) and b) of Lemma \ref{lemma3.1}. Therefore, the conditions c) and d) of Lemma \ref{lemma3.1} are fulfilled.

    Now let's consider the function \( g(u) \) defined by
	\[
	g(u) = 
	\begin{cases}
		f(2u), & \text{for } u \in \left[0, \tfrac{1}{2}\right), \\
		0, & \text{for } u \in \left[\tfrac{1}{2}, 1\right).
	\end{cases}
	\]
	This construction essentially compresses the behavior of the function \( f \) into the first half of the interval \( [0,1] \), and assigns the value zero to the remaining portion.
	
	Taking the derivative of \( g(u) \), we obtain:
	\[
	g'(u) = 
	\begin{cases}
		4\pi \sin\left(4\pi(u - \tfrac{1}{2})\right), & \text{for } u \in \left[0, \tfrac{1}{2}\right), \\
		0, & \text{for } u \in \left[\tfrac{1}{2}, 1\right).
	\end{cases}
	\]
	Thus, \( g'(u) \) is well-defined almost everywhere on \( [0,1] \) and represents a piecewise smooth function.

	Finally, let us suppose that \( g = \gamma \) and \( \varphi_n = \vartheta_n \), where \( \gamma \) and \( \vartheta_n \) are the function and system considered in the framework of Theorem \ref{theorem3}. Under these assumptions, we conclude that all the conditions stated in Theorem \ref{theorem3} are satisfied.
	
	As a consequence, we arrive at the following divergence result:
\begin{align}\label{eq:3.18}
	\lim_{n \to \infty} \sup \left| \sigma_n^\alpha(x, \gamma, \vartheta) \right|
	&= \limsup_{n \to \infty} \sum_{k=1}^n C_n(\gamma, \vartheta) \vartheta_n(x) \notag \\
	&= +\infty, \quad \text{a.e. } x \in \left[0, \tfrac{1}{2} \right]. 
\end{align}

	Without loss of generality, we can extend the validity of \eqref{eq:3.16} to hold almost everywhere on the entire interval \( [0,1] \).
	
	Moreover, since \( \gamma' = g' \in \mathrm{Lip}_1 \), i.e., the derivative of \( \gamma \) belongs to the Lipschitz class with exponent 1, the regularity condition required by Theorem \ref{theorem3} is fulfilled. Therefore, we can conclude that Theorem \ref{theorem3} is completely proved under the given assumptions.
	
\end{proof}

\section{Problems of efficiency}
\begin{theorem}
	\label{theorem4.1} Let us define the sequence of functions \(\{\varphi_n\}\) on the interval \([0,1]\) by
	\[
	\varphi_n(u) = \sqrt{2} \cos(2\pi n u).
	\]
	
	 Then for any \(x \in [0, 1]\) (see \cite{kashin1999}, Ch. 4), the following estimate holds:
	 \[
	 H_n(x) = O(1).
     \]
	
\end{theorem}

\begin{proof}
By the Cauchy inequality and the fact that \( D_{nk}^\alpha < 1 \), we estimate:

\begin{align*}
	T_n(x) &= \frac{1}{n} \sum_{i=1}^{n-1} \left| \int_0^{\frac{i}{n}} Q_n(u,x) \, du \right| \notag \\
	&= \frac{2}{n} \sum_{i=1}^{n-1} \left| \int_0^{\frac{i}{n}} \sum_{k=1}^n D_{nk}^{\alpha} \left( \int_0^u \cos(2\pi k v) \, dv \right) \cos(2\pi k x) \, du \right| \notag \\
	&= O(1) \cdot \max_{1 \leq i \leq n} \left| \int_0^{\frac{i}{n}} \frac{1}{2\pi} \sum_{k=1}^n D_{nk}^\alpha \cdot \frac{1}{k} \sin(2\pi k u) \cos(2\pi k x) \, du \right| \notag \\
	&= O(1) \cdot \left( \int_0^1 \left( \sum_{k=1}^n \frac{1}{k^2} \right) \, du \right)^{\frac{1}{2}} = O(1).
\end{align*}

Which means that the proof of Theorem \ref{theorem4.1} is valid.

\end{proof}

It appears that our findings apply to other orthonormal systems. In Theorem \ref{theorem4.1}, we have discussed and established that our results hold for trigonometric systems; now we will demonstrate their validity for Haar systems.

\begin{theorem}
	\label{theorem4.2}
	Let \( (X_n) \) be the Haar system (see \cite{kashin1999}, Ch. 3). Then for any \( x \in [0,1] \), 
	\[
	H_n(x) = O(1).
	\]
\end{theorem}

\begin{proof}
	It is known that when \( 2^s < m \leq 2^{s+1} \), then for all \( u \in [0,1] \),
	\[
	\left| \int_0^u X_m(v) \, dv \right| \leq 2^{-\frac{s}{2}}.
	\]
	From the above estimate, it follows that for any \( i = 1, 2, \dots, n-1 \) and any \( x \in [0,1] \), we have
	\[
	\left| \int_0^{\frac{i}{n}} \int_0^u \sum_{m=2^s+1}^{2^{s+1}} \left( D_{nk}^{\alpha} X_m(v) \right) X_m(x) \, dv \, du \right| \leq 2 \cdot 2^{-s}.
	\]
	
	Thus, we obtain the following estimate:
	\begin{align*}
		T_n(x) 
		&= \frac{1}{n} \sum_{i=1}^{n-1} \left| \int_0^{\frac{i}{n}} Q_n(u, x) \, du \right| \notag \\
		&= O(1) \cdot \max_{1 \leq i < n} \left| \int_0^{\frac{i}{n}} \sum_{s=0}^{d} \sum_{m=2^s+1}^{2^{s+1}} \left( D_{nk}^{\alpha} X_m(v) \right) X_m(x) \, dv \right| \notag \\
		&= O(1).
	\end{align*}
	
	Finally, we can conclude that Theorem \ref{theorem4.2} is valid.
	
\end{proof}

\end{document}